\documentclass[preprint,12pt]{elsarticle}

\usepackage{graphicx}
\usepackage{amssymb}
\usepackage{amsmath}
\usepackage{mathtools}
\usepackage{seqsplit}
\usepackage[utf8]{inputenc}

\usepackage{lineno}
\usepackage{amsthm}
\usepackage{dsfont}
\usepackage{hyperref}
\usepackage{tikz}
\usetikzlibrary{positioning}
  \def\gn#1#2{{$\href{http://groupnames.org/\#?#1}{#2}$}}
\def\gn#1#2{$#2$}  
\tikzset{sgplattice/.style={inner sep=1pt,norm/.style={red!50!blue},char/.style={blue!50!black},
  lin/.style={black!50}},cnj/.style={black!50,yshift=-2.5pt,left=-1pt of #1,scale=0.5,fill=white}}

\newcommand{\Ind}{\operatorname{Ind}}
\newcommand{\RO}{\operatorname{RO}}
\newcommand{\SL}{\operatorname{SL}}

\newcommand{\GL}{\operatorname{GL}}

\newcommand{\Aut}{\operatorname{Aut}}
\newcommand{\tr}{\operatorname{tr}}
\newcommand{\PSm}{\operatorname{PSm}}
\newcommand{\Sm}{\operatorname{Sm}}
\newcommand{\Res}{\operatorname{Res}}
\newcommand{\Diff}{\operatorname{Diff}}

\newcommand{\SmallGroup}{\operatorname{SmallGroup}}
\newcommand{\Mod}[1]{\ (\mathrm{mod}\ #1)}
\newcommand{\R}{\operatorname{R}}
\newcommand{\rank}{\operatorname{rank}}
\newcommand{\PO}{\operatorname{PO}}
\newcommand{\prim}{\operatorname{prim}}
\newcommand\blfootnote[1]{%
  \begingroup
  \renewcommand\thefootnote{}\footnote{#1}%
  \addtocounter{footnote}{-1}%
  \endgroup
}
\newtheorem{theorem}{Theorem}[section]
\newtheorem{definition}[theorem]{Definition}
\newtheorem{lemma}[theorem]{Lemma}
\newtheorem{corollary}[theorem]{Corollary}

\newtheorem{proposition}[theorem]{Proposition}
\newtheorem{remark}[theorem]{Remark}
\newtheorem{conjecture}[theorem]{Conjecture}

\newtheorem*{theorem*}{Theorem}
\newtheorem*{question*}{Question}
\newtheorem{theoremx}{Theorem}

\newtheorem{conjecturex}[theoremx]{Conjecture}

\journal{Journal Name}

\def\acts{\curvearrowright}
\makeatletter
\def\ps@pprintTitle{%
 \let\@oddhead\@empty
 \let\@evenhead\@empty
 \def\@oddfoot{}%
 \let\@evenfoot\@oddfoot}
\makeatother

\DeclareMathOperator{\Ima}{Im}

\begin{document}
\begin{frontmatter}


\title{\textbf{Inducing of exotic smooth two fixed point actions on spheres}}



\author{Piotr Mizerka}

\begin{abstract}
This paper is concerned with the Smith question \cite{Smith1960} which reads as follows. Is it true that for a finite group acting smoothly on a sphere with exactly two fixed points, the tangent spaces at the fixed points have always isomorphic group module structures defined by differentiation of the action? We show that one can answer this question negatively by using the technique of induction of group representations. We apply our results to indicate new dimensions of spheres admitting actions of specific Oliver groups, which give the negative answer to the Smith question. In particular, for the first time, we indicate some solvable non-nilpotent Oliver groups which yield negative answers to the Smith question.
\end{abstract}




\end{frontmatter}

\linenumbers
\setcounter{section}{0}
\nolinenumbers
\section*{Introduction}
\blfootnote{
$2010$ \emph{Mathematics Subject Classification.} Primary $57S25$; Secondary $55M35$.

\emph{\textcolor{white}{ggf}Keywords and phrases.} induced character, fixed point, Smith problem, smooth action.}
In this article, we are concerned with the Smith question for finite groups \cite[p. 406, the footnote]{Smith1960}. 
\begin{question*}[Smith question]
Is it true that for a finite group acting smoothly on a sphere with exactly two fixed points, the tangent spaces at the fixed points have always isomorphic group module structures defined by differentiation of the action?
\end{question*}
In the case where for the action in the Smith question, the group modules are not isomorphic to each other, we say that the action is \emph{Smith exotic}. In $1968$, Atiyah and Bott \cite{AtiyahBott1968} gave the affirmative answer to this question for cyclic groups of prime order. Petrie and his students and collaborators worked out a program which leads to obtaining negative answers to the Smith question, see \cite{Petrie1978},\cite{Petrie1979},\cite{Petrie1982} and \cite{Petrie1983}. Also, Cappell and Shaneson \cite{CappellShaneson1980}\cite{Cappell1982} obtained similar answers in the case the acting group $G=C_{4k}$, $k\geq 2$, the cyclic group of order divisible by $4$ and at greater or equal $8$. In fact, according to the current state of knowledge, the action of $C_8$ on $S^9$ is an example of a Smith exotic action with the smallest dimension of the sphere. On the other hand, Bredon \cite{Bredon1969} proved that if $G=C_{2^n}$, the cyclic group of order $2^n$, then there exists a treshold dimension $D\geq 0$ with the property that there does not exists a Smith exotic action of $G$ on $S^n$ whenever $n\geq D$. A lot of further research on the Smith problem, including Illman \cite{Illman1982}, Milnor \cite{Milnor1966} was done. Later, results concerning the Smith sets (definition in the notation section) were obtained, see the works of Laitinen, Morimoto, Pawałowski, Solomon and Sumi, (\cite{LaitinenPawalowski1999}, \cite{PawalowskiSolomon2002}, \cite{MorimotoPawalowski2003}, \cite{Sumi2016}). For a comprehensive survey on the Smith problem, we refer the reader to the work of Pawałowski \cite{Pawalowski2018}. To our knowledge, the question of examining the dimensions of \emph{Smith exotic spheres} (that is spheres admitting Smith exotic actions) has not been studied intensively yet. Before we proceed, let us establish some simplified notation. Assume a finite group $G$ acts smoothly on a manifold $M$ with a fixed point $x\in M$. We shall refer to the $\mathbb{R}G$-module structure at the tangent space $T_xM$, defined by differentiation of the action of $G$ on $M$, simply as an \emph{$\mathbb{R}G$-module structure at the fixed point $x$}.

We will need the notion of a \emph{Oliver group}. This is a group $G$ with the property that there does not exist a sequence of subgroups $P\trianglelefteq H\trianglelefteq G$ ($P$ not necessarily normal in $G$) such that $P$ and $G/H$ are of prime power orders and $H/P$ is cyclic. By the work of Oliver \cite{Oliver1996} and Morimoto \cite{Morimoto19982} we know that Oliver groups are precisely the finite groups admitting smooth fixed point free actions on disks. Using this property, one can easily deduce that if $G$ is a finite group with an Oliver subgroup $H$, then $G$ must be Oliver as well. Indeed, in such a case there exists a smooth fixed point free action of $H$ on a disk $D$ without fixed points. Then, by the induced action, $G$ acts on $D^{[G:H]}$, which is diffeomorphic to a disk. Moreover, the fixed point set for the induced $G$-action is preserved and thus empty. Given an Oliver group $G$ and its normal subgroup $N$, we focus deducing Smith exotic actions of $G$ using the properties of the induction homomorphism $\Ind_N^G$. For any subgroup $H$ of a finite group $G$, we provide necessary and sufficient conditions for $\Ind_H^G$ to be a monomorphism - defined both for $\R(H)$, the representation group in the general complex case, and for $\RO(H)$, the representation group in the real case. This condition depends only on the conjugacy classes of $H$ and $G$. For a better presentation of this result, let us define a \emph{real conjugacy class} of an element $g\in G$ to be the set $(g)^{\pm}=(g)\cup(g^{-1})$ and, for any $h\in H$, denote by $(h)_H$ and $(h)_G$ the conjugacy classes of $h$ in $H$ and $G$ respectively. Analogously, $(h)_H^{\pm}=(h)_H\cup(h^{-1})_H$ and $(h)_G^{\pm}=(h)_G\cup(h^{-1})_G$. Then, the condition can be stated as follows.

\begin{theoremx}\label{theorem:A}
Let $G$ be a finite group and $H$ its a subgroup. Then the following holds.
\begin{enumerate}[(1)]
    \item $\Ind_H^G:\R(H)\rightarrow\R(G)$ is a monomorphism if and only if $(h)_G\cap H=(h)_H$ for any $h\in H$.
    \item $\Ind_H^G:\RO(H)\rightarrow\RO(G)$ is a monomorphism if and only if $(h)_G^{\pm}\cap H=(h)_H^{\pm}$ for any $h\in H$.
\end{enumerate}
\end{theoremx}

\noindent We apply the theorem above to determine for which prime powers $q$ the induction homomorphism $\Ind_H^G$ is a monomorphism in the case $H=\SL(2,q)$ and $G=\GL(2,q)$, the special and general linear groups of $(2\times 2)$-matrices with entries in the field $\mathbb{F}_q$ with $q$ elements. 
 
 Further, we prove the existence of Smith exotic actions using the induced modules. This can be summarized in the following theorem in which we use the notion of \emph{Smith matched} $\mathbb{R}G$-modules defined in section \ref{section:conditionsSufficiente} and the notion of \emph{$\mathcal{P}$-orientability} defined in the notation section.
\begin{theoremx}\label{theorem:B}
Let $G$ be an Oliver group and $N$ its normal subgroup such that $\Ind_N^G:\RO(N)\rightarrow\RO(G)$ is a monomorphism. If two non-isomorphic $\mathbb{R}N$-modules $U$ and $V$ are Smith matched and $\Ind_N^G(U)$ and $\Ind_N^G(V)$ are $\mathcal{P}$-oriented, then there exists a Smith exotic action of $G$ on a standard sphere with $\mathbb{R}G$-module structures at the two fixed points isomorphic to $\Ind_N^G(U)$ and $\Ind_N^G(V)$. \footnote{According to a private information of Masaharu Morimoto, in the equivariant surgery arguments used to get the result, we do not need to assume that the $\mathcal{P}$-orientability condition for $\Ind_N^G(U)$ and $\Ind_N^G(V)$ holds.}
\end{theoremx}
In section \ref{section:conditionsSufficiente}, we point out that once we have two non-isomorphic Smith matched and $\mathcal{P}$-oriented $\mathbb{R}G$-modules $U$ and $V$, we can construct a Smith exotic action on a sphere with $\mathbb{R}G$-module structures at the fixed points isomorphic to $U$ and $V$. This conclusion and Theorem \ref{theorem:B} motivate us to state the following conjecture.
\begin{conjecturex}\label{conjecture:C}
Assume $G$ is an Oliver group and $N$ its normal subgroup such that $\Ind_N^G:\RO(N)\rightarrow\RO(G)$ is a monomorphism. Then, if $N$ admits Smith exotic actions, so does $G$.
\end{conjecturex}
\noindent We wish to recall the Dovermann-Suh Conjecture \cite[p. 44]{Dovermann1992} which asserts that for any Oliver group $G$ and its subgroup $H$, the negative answer to the Smith question for $H$ yields the same answer for $G$. Counterexamples to their conjecture are given in \cite{Pawalowski2018}. One may consider Conjecture \ref{conjecture:C} as a modified version of the Dovermann-Suh Conjecture, which we hope is true.

The paper is organized as follows. First, we fix the notation used to prove the main results. Further contents is divided into three parts. In the first we prove Theorem \ref{theorem:A} and give its application for $\SL(2,q)\trianglelefteq\GL(2,q)$. The second part concerns Theorem \ref{theorem:B}. In the first paragraph we list sufficient conditions for exotic constructions we are interested in. Next, we deal with ensuring these conditions for induced actions and prove Theorem \ref{theorem:B}. At the end, we present an application to the case of the direct product of $\SL(2,5)$ and an arbitrary finite group. The third part provides examples of groups for which Conjecture \ref{conjecture:C} holds.
\section*{Notation}
All groups considered in this paper are assumed to be finite. By an \emph{action} of a group $G$ on a smooth manifold $M$ we mean a homomorphism $\varphi:G\rightarrow\Diff(M)$, where $\Diff(M)$ denotes the group of all diffeomorphisms of $M$. Instead of writing formally $\varphi(g)(x)$ for a value of a diffeomorphism $\varphi(g)$ at point $x\in M$, we use the notation $gx$ when $\varphi$ can be inferred from the context. We also use the notion of a \emph{linear action} (we shall call it by \emph{action} as well) of $G$ on an $n$-dimensional vector space $V$ over the field $F$. Such an action is a homomorphism $\varphi:G\rightarrow\GL(n,F)$. As before, we use an abbreviation $gv$ for $\varphi(g)(v)$ for $g\in G$ and $v\in V$. The linear action of $G$ on $V$ endows it with a structure of an $n$-dimenensional $FG$-module (see \cite{Liebeck2001}).

We recall the notions from \cite{MorimotoPawalowski2003} which we will need later. Given a group $G$ and a field $F$, the set of $FG$-modules admits a structure of a semigroup with the operation of direct sum. Applying the Grothendieck construction to this semigroup, we obtain a \emph{representation group} of $G$ over the field $F$. We use the notation $U-V$ for the element of the representation group of $G$ determined by the equivalence class of a pair $(U,V)$ of $FG$-modules. Recall that, if $F$ is of characteristic $0$, then the representation group is a torsion-free abelian group. If $F=\mathbb{C}$, we use the symbol $\R(G)$ for the representation group of $G$, while if $F=\mathbb{R}$, we denote this group by $\RO(G)$. Let $\alpha(G)$ and $\beta(G)$ be the numbers of conjugacy and real conjugacy classes of $G$ respectively. Note that $\rank(\R(G))=\beta(G)$ and $\rank(\RO(G))=\beta'(G)$.

The \emph{primary group} of $G$ \cite{Pawalowski2018}, denoted by $\PO(G)$, is the subgroup of all elements $U-V\in \RO(G)$ such that $\Res_P^G(U)\cong\Res_P^G(V)$ for any subgroup $P\leq G$ of prime power order, i.e. the restrictions of $U$ and $V$ to subgroups of $G$ of prime power order are isomorphic. The \emph{reduced primary group} of $G$, denoted by $\widetilde{\PO}(G)$, consists of all elements $U-V\in \PO(G)$ such that $U^G=V^G=\{0\}$. The primary number of $G$ which we denote by $\prim(G)$ is the number of real conjugacy classes of $G$ containing elements which order is divisible by at least two different primes. The rank of $\PO(G)$ equals then $\prim(G)$, while the rank of $\widetilde{\PO}(G)$ is $\prim(G)-1$, see \cite[Lemma 2.1]{LaitinenPawalowski1999}. The \emph{Smith set}, $\Sm(G)$, of a group $G$ is the set of elements $U-V\in\RO(G)$ such that there exists a two fixed point action on a homotopy sphere $\Sigma$ with the $\mathbb{R}G$-modules structures at the two fixed points isomorphic to $U\oplus W$ and $V\oplus W$ for some $\mathbb{R}G$-module $W$. If $U-V\in\Sm(G)$ for some $\mathbb{R}G$-modules $U$ and $V$, then we say that $U$ and $V$ are \emph{Smith equivalent}. Using the surgery theory developed by Morimoto \cite{Morimoto1998}, the word 'homotopy sphere' in the definition of the Smith set can be replaced by 'standard sphere'. The primary Smith set of $G$, $\PSm(G)$, is defined as the intersection $\Sm(G)\cap\widetilde{\PO}(G)$. 
We also need the notion of a \emph{pseudocyclic group}. A group $G$ is called pseudocyclic if for some prime $p$, $G$ contains a normal $p$-subgroup $P$ such that $G/P$ is cyclic. Obviously, all pseudocyclic groups are not Oliver groups. We denote by $\mathcal{PC}(G)$ the family of all pseudocyclic subgroups of $G$. For a given group $G$, let us denote by $\mathcal{P}(G)$ the family of all subgroups of $G$ which are of prime power order.

For a prime $p$, let us use the notation $\mathcal{O}^p(G)$ for the smallest normal subgroup of $G$ such that the quotient of $G$ by this subgroup is a $p$-group. A subgroup $H$ of a group $G$ is called \emph{large} if $\mathcal{O}^p(G)\leq H$ for some prime $p$. We denote by $\mathcal{L}(G)$ the family of all large subgroups of $G$. We say that an $\mathbb{R}G$-module $V$ satisfies the $\emph{weak gap condition}$ if for any subgroup $P\in\mathcal{P}(G)$ and $P<H\leq G$, we have $\dim V^P\geq 2\dim V^H$. Moreover, we say that $V$ is \emph{$\mathcal{P}$-oriented} if $V^P$ is oriented as a vector space for any $P\in\mathcal{P}$ and the transformation $\theta_g:V^P\rightarrow V^P, v\mapsto gv$ preserves the orientation for any $g\in N_G(P)$, where $N_G(P)$ denotes the normalizer of $P$ in $G$.
\section{Induction homomorphisms}
\subsection{Proof of Theorem \ref{theorem:A}}
Assume $H$ is a subgroup of a group $G$. For the simplicity, let us denote by $\Ind_{\mathbb{C}}$ the induction homomorphism $\Ind_H^G:\R(H)\rightarrow\R(G)$ and by $\Ind_{\mathbb{R}}$ its restriction to $\RO(H)$. Define $a$ and $a'$ to be the numbers of conjugacy and real conjugacy classes of $G$ respectively which have non-zero intersection with $H$ and put $b=\beta(H)$ and $b'=\beta'(H)$.

We will need the recipe for computing induced characters.
\begin{theorem}\label{theorem:inducedCharacterFormula}\cite[21.23. Theorem]{Liebeck2001}
Let $\chi$ be a character of $H$ and $g\in G$. Then, we have two possibilities.
\begin{enumerate}[(1)]
    \item if $H\cap (g)=\emptyset$, then $\Ind_{\mathbb{C}}(\chi)(g)=0$.
    \item if $H\cap (g)\neq\emptyset$, then
    $$
    \Ind_{\mathbb{C}}(\chi)(g)=|C_G(g)|\Big(\frac{\chi(h_1)}{|C_H(h_1)|}+\ldots+\frac{\chi(h_m)}{|C_H(h_m)|}\Big),
    $$
\end{enumerate}
where $C_K(x)$ denotes the centralizer of the element $x$ of the group $K$ and $h_1,\ldots,h_m$ are the representatives of all the distinct conjugacy classes in $H$ of the elements of the set $H\cap (g)$.
\end{theorem}
\begin{lemma}\label{lemma:boundForRank}
$\rank(\Ima(\Ind_{\mathbb{C}}))\leq\min(a,b)$ and $\rank(\Ima(\Ind_{\mathbb{R}}))\leq\min(a',b')$.
\end{lemma}
\begin{proof}
Let $C$ be the matrix of $\Ind_{\mathbb{C}}$ (we take the standard bases of $\R(H)$ and $\R(G)$ consisting of complex irreducible characters). Assume $\chi_1,...,\chi_n$ are complex irreducible characters of $H$ and $\psi_1,...,\psi_m$ are complex irreducible characters of $G$ (note that $b=n$). Let
$$
\Ind_{\mathbb{C}}(\chi_i)=\sum_{k=1}^mb_{ik}\psi_k=(a_{i1},...,a_{im})
$$
where $a_{ij}=\Ind_{\mathbb{C}}(\chi_i)(g_j)$ and $G=(g_1)\cup...\cup(g_m)$ is the decomposition of $G$ into conjugacy classes with representatives $g_j$, $1\leq j\leq m$. Let $\psi_j=(\psi_{j1},...,\psi_{jm})$, where $\psi_{jl}=\psi_j(g_l)$. Then $A=BX$, where
$$
\begin{tabular}{ccc}
$
A=\begin{pmatrix}
    a_{11}       &  \dots & a_{1m} \\
    \hdotsfor{3} \\
    a_{n1}       &\dots & a_{nm}
\end{pmatrix}
$
&
$
B=\begin{pmatrix}
    b_{11}       &  \dots & b_{1m} \\
    \hdotsfor{3} \\
    b_{n1}       &\dots & b_{nm}
\end{pmatrix}
$
&
$
X=\begin{pmatrix}
    \psi_{11}       &  \dots & \psi_{1m} \\
    \hdotsfor{3} \\
    \psi_{m1}       &\dots & \psi_{mm}
\end{pmatrix}
$
\end{tabular}
$$
and $b_{ij}$ are non-negative integers for $1\leq i\leq n$ and $1\leq j\leq m$. Then $C=B^T$, since 
\begin{center}
    $C(\chi_i)=\begin{pmatrix}
    b_{i1} \\
    \vdots \\
    b_{in}
\end{pmatrix}$
\end{center}
Thus $C=(AX^{-1})^T$ and 
\begin{align*}
    \rank(C)&=\rank(AX^{-1})=\rank(AX)\leq\min(\rank(A),\rank(X))\\
    &=\min(\rank(A),m)\leq\rank(A)
\end{align*}
It follows that $\rank(A)\leq\min(a,n)$. Indeed, from Theorem \ref{theorem:inducedCharacterFormula} we conclude that $a_{ij}=0$ whenever $(g_j)\cap H=\emptyset$. Thus, since the characters of $G$ are constant on the conjugacy classes of $G$, we have $\rank(A)\leq a$. Obviously $\rank(A)\leq n$ as $n$ is the dimension of $\R(H)$. Therefore, since $b=n$, we have $\rank(C)\leq\min(a,b)$. Since $C$ is an integer matrix and $\mathbb{C}$ is an extension of $\mathbb{Q}$, we conclude that the complex rank of $C$ equals its rational rank, that is $\rank(C)=\rank_{\mathbb{Q}}(C)$. Now, denote by $r$ and $r'$ the ranks of $\Ima(\Ind_{\mathbb{C}})$ and $C$ respectively. We show that $r=r'$ which would mean that $r\leq\min(a,b)$ and would complete the proof. Obviously, $r\geq r'$. Let $V$ be the vector space over $\mathbb{Q}$ spanned by the generators of $\R(G)$. Take any $r'+1$ elements $v_1,...,v_{r'+1}$ from $\Ima(\Ind_{\mathbb{C}})$. They can be considered as vectors from $V$. Note that they are linearly dependent, since the dimension of $\Ima(\Ind_{\mathbb{C}})$ considered as a subspace of $V$ equals $r'$. Let
\begin{equation}\label{equation:linearCombination}
\alpha_1v_1+...+\alpha_{r'+1}v_{r'+1}=0
\end{equation}
be a nontrivial combination. Suppose $\{\alpha_{i_1},...,\alpha_{i_k}\}$ is the set of all nonzero coefficients and $\alpha_{i_j}=p_j/q_j$, where $p_j,q_j\in\mathbb{Z}\setminus{\{0\}}$ for $1\leq j\leq k$. Multiplying both sides of equality (\ref{equation:linearCombination}) by $q_1...q_k$, we get a nontrivial integer combination of $v_j$'s. Thus $r\leq r'$ and, as a result $r=r'$.

The proof of part $(2)$ is analogous.
\end{proof}
\begin{lemma}\label{lemma:trace}
Assume $A\in\GL(n,\mathbb{C})$ is of finite order. Then $\tr(A^{-1})=\overline{\tr(A)}$.
\end{lemma}
\begin{proof}
We may assume that $A$ is already in its Jordan canonical form (since conjugation in $\GL(n,\mathbb{C})$ is trace-invariant). If $\lambda_1,...,\lambda_n$ are the eigenvalues of $A$, then $\lambda_1^{-1},...,\lambda_n^{-1}$ are the eigenvalues of $A^{-1}$, for $A$ is upper-triangular. Note that $\lambda_i$ is a $k$-th root of unity for $i=1,...,n$. Indeed, if $v\in\mathbb{C}^n$ is an eigenvector corresponding to $\lambda_i$, then
$$
v=I_nv=A^kv=\lambda_i^kv
$$
Thus $\lambda_i^k=1$ since $v$ is non-zero. Thus $\lambda_i^{-1}=\overline{\lambda_i}$ for $i=1,...,n$ and, as a result,
$$
\tr(A^{-1})=\lambda_1^{-1}+...+\lambda_n^{-1}=\overline{\lambda_1}+...+\overline{\lambda_n}=\overline{\lambda_1+...+\lambda_n}=\overline{\tr(A)}
$$
\end{proof}
\begin{proof}[Proof of Theorem \ref{theorem:A}]
Let us prove the first part. Assume that for any $h\in H$ its conjugacy class in $G$ coincides with its conjugacy class in $H$, that is $(h)_G=(h)_H$. Let $\chi$ be a character of $H$. Pick $g\in G$. If $H\cap (g)=\emptyset$, then, by Theorem \ref{theorem:inducedCharacterFormula}, we get $\Ind_{\mathbb{C}}(\chi)(g)=0$. Otherwise, if $h\in H\cap(g)$, then 
$$
\Ind_{\mathbb{C}}(\chi)(g)=\frac{|C_G(h)|}{|C_H(h)|}\chi(h)
$$
Thus, if $\chi\neq\psi$ are complex characters of $H$ with $\chi(h)\neq\psi(h)$ for some $h\in H$, then
$$
\Ind_{\mathbb{C}}(\chi)(h)=\frac{|C_G(h)|}{|C_H(h)|}\chi(n)\neq\frac{|C_G(h)|}{|C_H(h)|}\psi(n)=\Ind_{\mathbb{C}}(\psi)
$$
Therefore $\Ind_{\mathbb{C}}$ is injective.

Let us prove now the converse implication. Suppose $\Ind_{\mathbb{C}}$ is a monomorphism. For the converse, assume that there exists $h\in H$ with $(h)_H\neq(h)_G\cap H$. Then $(h)_H\subsetneq(h)_G\cap H$ and therefore $a<b$. Thus, by Lemma \ref{lemma:boundForRank}, $\rank(\Ima(\Ind_{\mathbb{C}}))<b$, so $\Ind_{\mathbb{C}}$ cannot be injective. A contradiction.

We show the second part now. Suppose that for any $h\in H$ we have $(h)^{\pm}_G\cap H=(h)^{\pm}_H$. Let $\chi$ and $\psi$ be two different characters of $H$. We must show that $\Ind_{\mathbb{R}}(\chi)\neq\Ind_{\mathbb{R}}(\psi)$. Take $h\in H$ with $\chi(h)\neq\psi(h)$. We have two possibilities. The first one is when $(h)_H=(h^{-1})_H=(h)_G^{\pm}\cap H$. Thus 
$$(h)_G\cap H=(h^{-1})_G\cap H=(h)^{\pm}_G\cap H=(h)^{\pm}_H=(h)_H$$
and, as in the proof of the first part, we get
$$
\begin{tabular}{ccc}
$
\Ind_{\mathbb{R}}(\chi)(h)=\frac{|C_G(h)|}{|C_H(h)|}\chi(h)
$
&
and
&
$
\Ind_{\mathbb{R}}(\psi)(h)=\frac{|C_G(h)|}{|C_H(h)|}\psi(h)
$
\end{tabular}
$$
 Therefore $\Ind_{\mathbb{R}}(\chi)(h)\neq\Ind_{\mathbb{R}}(\psi)(h)$ since $\chi(h)\neq\psi(h)$. In the second possibility, we have $(h)_H\neq(h^{-1})_H$. If $(h)_G\cap H=(h)_H$, we have already proved the assertion. Otherwise, since $(h)^{\pm}_G\cap H=(h)^{\pm}_H$, we conclude that $(h)_G\cap H=(h^{-1})_G$ (otherwise $(h)_G\cap H=(h)_H$ and $(h^{-1})_G\cap H=(h^{-1})_H$ which we have already considered). Thus $(h)_G=(h)_H\cup(h^{-1})_H$. Note that $|C_H(h)|=|C_H(h^{-1})|$. Thus by Theorem \ref{theorem:inducedCharacterFormula} and Lemma \ref{lemma:trace}, we get 
 $$
\begin{tabular}{ccc}
$
\Ind_{\mathbb{R}}(\chi)(h)=2\frac{|C_G(h)|}{|C_H(h)|}\chi(h)
$
&
and
&
$
\Ind_{\mathbb{R}}(\psi)(h)=2\frac{|C_G(h)|}{|C_H(h)|}\psi(h)
$
\end{tabular}
$$
Thus $\Ind_{\mathbb{R}}(\chi)(h)\neq\Ind_{\mathbb{R}}(\psi)(h)$.

We prove now the converse. Suppose $\Ind_{\mathbb{R}}$ is a monomorphism. Assume for the contrary that there exists $h\in H$ with $(h)^{\pm}_G\cap H\neq(h)^{\pm}_H$. Then $(h)^{\pm}_H\subsetneq(h)^{\pm}_G\cap H$ and thus $a'<b'$. Hence, it follows by Lemma \ref{lemma:boundForRank} that $\rank(\Ima(\Ind_{\mathbb{R}}))<b'$ and $\Ind_{\mathbb{R}}$ is not injective which is a contradiction with our assumption.
\end{proof}
\begin{corollary}\label{corollary:inducedNormal}
Assume $N$ is a normal subgroup of $G$. Then the following holds.
\begin{enumerate}[(1)]
    \item $\Ind_N^G:\R(N)\rightarrow\R(G)$ is a monomorphism if and only if the conjugacy classes in $G$ of elements from $N$ coincide with their conjugacy classes in $N$.
    \item $\Ind_N^G:\RO(N)\rightarrow\RO(G)$ is a monomorphism if and only if the real conjugacy classes in $G$ of elements from $N$ coincide with their real conjugacy classes in $N$.
\end{enumerate}
\end{corollary}
\subsection{Application to the case $\SL(2,q)\trianglelefteq\GL(2,q)$}
Put $N=\SL(2,q)$ and $G=\GL(2,q)$ where $q$ is a prime power. Consider the following elements of $N$.
\begin{center}
\begin{tabular}{ll*{1}{c}r}
$\mathbf{1}=\begin{pmatrix}1&0\\0&1\end{pmatrix}$     
& $z=\begin{pmatrix}-1&0\\0&-1\end{pmatrix}$ 
&$c=\begin{pmatrix}1&0\\1&1\end{pmatrix}$ \\
$d=\begin{pmatrix}1&0\\\nu&1\end{pmatrix}$     
&$a=\begin{pmatrix}\nu&0\\0&\nu^{-1}\end{pmatrix}$ 
&$b$ \\
\end{tabular}
\end{center}
where $\nu$ is a generator of $\mathbb{F}_q^{\times}$ and $b$ is an element of order $q+1$. The conjugacy classes of $N$ can be characterized as follows (see \cite[Theorem 4.1, Theorem 6.1]{Braun2016}).
\begin{lemma}\label{lemma:conjugacyClassesSL2q}
If $q=2^k$, $k\geq 1$, then the conjugacy classes of $N$ are $(\mathbf{1})$, $(c)$, $(a^l)$, $(b^m)$, where $1\leq l\leq (q-2)/2$ and $1\leq m\leq q/2$. In the case $q$ is odd, the conjugacy classes of $N$ are $(\mathbf{1})$, $(z)$, $(c)$, $(d)$, $(zc)$, $(zd)$, $(a^l)$, $(b^m)$, where $1\leq l\leq (q-1)/2$ and $1\leq m\leq (q-1)/2$.
\end{lemma}
For the next three lemmas, assume $q$ is a power of $2$. Let us recall the complex character table of $N$.
\begin{lemma}\label{lemma:characterTable}
The following table is the character table of $N$.
\begin{center}
\begin{tabular}{l|l*{2}{c}r}
            &  $\mathbf{1}$&  $c$ & $a^l$  & $b^m$ \\
\hline
$\mathds{1}$ & $1$ & $1$ & $1$ & $1$  \\
$\psi$   & $q$ & $0$ & $1$ & $-1 $ \\
$\chi_i$   & $q+1$ & $1$  & $\nu_{q-1}^{il}$ &$ 0$  \\
$\theta_j$    & $q-1$ & $-1$  &$ 0$ & $-\nu_{q+1}^{jm}$ 
\end{tabular}
\end{center}
where $1\leq l\leq (q-2)/2$, $1\leq m\leq q/2$ and $\nu_r^s=\zeta_r^s+\zeta_r^{-s}$ where $\zeta_r=\exp{(2\pi i/r)}$.
\end{lemma}
\begin{lemma}\label{lemma:permutationMatrix}
The conjugacy classes of $N$ constitute its real conjugacy classes.
\end{lemma}
\begin{proof}
Let $X$ be the complex character table of $N$ and $\overline{X}$ be the matrix conjugate to $X$. It follows by \cite[23.1 Theorem]{Liebeck2001} that there exists a parmutation matrix $P$ such that $\overline{X}=PX$ and the trace of $P$ is equal to the number of conjugacy classes of $N$ which elements are conjugate to their own inverses. However, by Lemma \ref{lemma:characterTable}, we conclude that $X$ has real-vealued entries and thus $P$ is the indentity matrix. Hence, the conjugacy classes of $N$ consitute its real conjugacy classes.
\end{proof}
\begin{lemma}\label{lemma:sl22nproof}
For any $n\in N$, we have $(n)_G=(n)_N$. Thus, $\Ind_{\mathbb{C}}$ and $\Ind_{\mathbb{R}}$ are monomorphisms.
\end{lemma}
\begin{proof}
Take $n\in N$ and consider $(n)_G$. It decomposes as the union $(n)_G=(n_1)\cup...\cup(n_k)_N$. Obviously $|n|=|n_1|=...=|n_k|$. Since $\mathbf{1}$ is the unique element of order $1$ in $N$, it follows that $(\mathbf{1})_G=(\mathbf{1})_N$. Similarly, $(c)_G=(c)_N$, since $(c)_N$ is the unique conjugacy class of $N$ which elements are order $2$.

Note that for $n\in N$ we have $(n)_G=(n)_N$ if and only if
$$
\frac{|C_G(n)|}{|C_N(n)|}=\frac{|G|}{|N|}
$$
Indeed, by the Orbit-Stabilizer Theorem,
$$
|(n)_G|=|(n)_N|\Leftrightarrow\frac{|G|}{|C_G(n)|}=\frac{|N|}{|C_N(n)|}
$$
If $q=2$, then there is nothing to prove since $N=G$. Assume $q\geq 4$. We show that for $n=a^l$, $1\leq l\leq(q-2)/2$ we have $(n)_G=(n)_N$. We have
\begin{align*}
\begin{pmatrix}
a&b\\c&d
\end{pmatrix}n=n\begin{pmatrix}
a&b\\c&d
\end{pmatrix}&\Leftrightarrow\begin{pmatrix}
a&b\\c&d
\end{pmatrix}\begin{pmatrix}
\nu^l&0\\0&\nu^{-l}
\end{pmatrix}=\begin{pmatrix}
\nu^l&0\\0&\nu^{-l}
\end{pmatrix}\begin{pmatrix}
a&b\\c&d
\end{pmatrix}\\
&\Leftrightarrow\begin{pmatrix}
a\nu&b\nu^{-l}\\c\nu&d\nu^{-l}
\end{pmatrix}=\begin{pmatrix}
a\nu&b\nu\\
c\nu^{-l}&d\nu^{-l}
\end{pmatrix}
\end{align*}
Thus $\begin{pmatrix}
a&b\\c&d
\end{pmatrix}\in C_G(n)$ if and only if $b=c=0$ and $a,d$ are non-zero, so $|C_G(n)|=(q-1)^2$, whereas $\begin{pmatrix}
a&b\\c&d
\end{pmatrix}\in C_N(n)$ if and only if $b=c=0$ and $ad=1$, so $|C_N(n)|=q-1$. Thus
$$
\frac{|G|}{|N|}=\frac{(q^2-1)(q^2-q)}{q(q^2-1)}=q-1=\frac{(q-1)^2}{q-1}=\frac{|C_G(n)|}{|C_N(n)|}
$$
and $(n)_G=(n)_N$.

Consider now $n=b^m$, $1\leq m\leq q/2$. Then, we conclude from \cite[28.4 Theorem]{Liebeck2001} that $|C_G(n)|=q^2-1$ and from \cite[2.2.3]{Reeder2008} that $|C_n(n)|=q+1$. Thus
$$
\frac{|C_G(n)|}{|C_N(n)|}=\frac{q^2-1}{q+1}=q-1=\frac{|G|}{|N|}
$$
and $(n)_G=(n)_N$.
\end{proof}
Let us return to the general case $q=p^k$ where $p$ is a prime. The following theorem gives necessary and sufficient conditions for $\Ind_{\mathbb{C}}$ and $\Ind_{\mathbb{R}}$ to be monomorphisms. It is worth noting that these conditions depend only on $q$.
\begin{theorem}\label{theorem:sl2qgl2q}
$\Ind_{\mathbb{C}}$ is a monomorphism if and only if $q$ is a power of $2$, while $\Ind_{\mathbb{R}}$ is a monomorphism if and only if $q$ is a power of $2$ or $q\equiv 3\Mod{4}$.
\end{theorem}
\begin{proof}
The part involving $q=2^k$ follows from Lemma \ref{lemma:sl22nproof}. Thus, we can assume $q$ is odd and we have to show the following two statements.
\begin{enumerate}[(1)]
    \item $\Ind_{\mathbb{C}}$ is not a monomorphism.
    \item $\Ind_{\mathbb{R}}$ is a monomorphism if and only if $q\equiv 3\Mod{4}$.
\end{enumerate}
The elements $c$ and $d$ are not conjugate in $N$, they are conjugate in $G$, however, since
$$
\begin{pmatrix}
a&0\\c&a\nu
\end{pmatrix}\begin{pmatrix}
1&0\\1&1
\end{pmatrix}=\begin{pmatrix}
1&0\\\nu&1
\end{pmatrix}\begin{pmatrix}
a&0\\c&a\nu
\end{pmatrix}
$$ for any $a\neq 0$. Thus $(c)_N\subsetneq(c)_G=(d)_G$ and $(1)$ follows by Corollary \ref{corollary:inducedNormal}.

Now, take $q\equiv 1\Mod{4}$. It follows from \cite[Lemma 2]{Mizerka2019_2} that $(c)_N^{\pm}=(c)_N\subsetneq(c)_G\subseteq(c)_G^{\pm}$ and $(2)$ follows by Theorem \ref{theorem:A}. Suppose $q\equiv 3\Mod{4}$. Obviously $(\mathbf{1})_G=(\mathbf{1})_N$ and $(z)_G=(z)_N$ (these are conjugacy classes containing only one element). Note that $c$ and $d$ are the only elements of $N$ of order $q$. Thus, we conclude from \cite[Lemma 2]{Mizerka2019_2} that 
$$
(d)_G^{\pm}=(c)_G^{\pm}=(c)_G=(d)_G=(c)_N\cup(d)_N=(c)_N\cup(c^{-1})_N=(c)_N^{\pm}=(d)_N^{\pm}
$$
Analogously, we prove that $(zc)_G^{\pm}=(zd)_G^{\pm}=(zc)_N^{\pm}=(zd)_N^{\pm}$. We can repeat the reasoning from the proof of Lemma \ref{lemma:sl22nproof} to show that $(a^l)_G=(a^l)_N$ for $1\leq(q-3)/2$ which yields $(a^l)_G^{\pm}=(a^l)_N^{\pm}$. Finally, take $n=b^m$ for an arbitrary $1\leq m\leq q/2$. It follows by \cite[2.1.3]{Reeder2008} that $|C_N(n)|=q+1$ and by \cite[28.4 Theorem]{Liebeck2001} that $|C_G(n)|=q^2-1$. Thus
$$
\frac{|C_G(n)|}{|C_N(n)|}=q-1=\frac{|G|}{|N|}
$$
and therefore $(n)_G=(n)_N$, so $(n)_G^{\pm}=(n)_N^{\pm}$. Thus,we have proved that for any $n\in N$ the real conjugacy classes of $n$ in $N$ and $G$ coincide which concludes the proof by Corollary \ref{corollary:inducedNormal}.
\end{proof}
\section{Existence of Smith exotic actions}\label{section:conditionsSufficiente}
Assume $G$ is a group. Let us introduce the following definition after \cite{Morimoto1998}.
\begin{definition}\label{definition:gapCondition}
An $\mathbb{R}G$-module $V$ is said to satisfy the $(a)$ \textbf{strong} or $(b)$ \textbf{weak} \textbf{gap condition} if for any $P<H\leq G$ we have $\dim V^P>2\dim V^H$ for $(a)$ and $\dim V^P\geq 2\dim V^H$ for $(b)$.
\end{definition}
Throughout this paragraph, we assume that $G$ is an Oliver group with a normal subgroup $N$ such that $\Ind_N^G:\RO(N)\rightarrow\RO(G)$ is a monomorphism. We prove here Theorem \ref{theorem:B}.
\subsection{Sufficient conditions for an exotic action}
\begin{definition}\label{definition:smithMatched}
We call two $\mathbb{R}G$-modules \textbf{Smith matched} if the following conditions are satisfied.
\begin{enumerate}[(1)]
\item $U-V\in\widetilde{\PO}(G)$,
    \item $U$ and $V$ satisfy the weak gap condition,
    \item $\dim W^P\geq 5$, $\dim W^H\geq 2$ for $P$ a $p$-subgroup of $G$ and $H\in\mathcal{PC}(G)$, $W=U,V$,
    \item every pseudocyclic subgroup of $G$ occurs as an isotopy subgroup of the actions $G\acts U$ and $G\acts V$,
    \item $U^L=V^L=\{0\}$ for any $L\in\mathcal{L}(G)$,
    \item $\dim U=\dim V\geq 6$.
\end{enumerate}
\end{definition}

From the proof of \cite[Theorem 0.4]{Oliver1996} it follows that if two $\mathbb{R}G$-modules $U$ and $V$ are Smith matched and $\mathcal{P}$-oriented, then there exists a two fixed point action of $G$ on a disk $D$ with tangent spaces at fixed points isomorphic at these points to $U$ and $V$. This action can be constructed to satisfy the following conditions.
\begin{enumerate}[(1')]
\item $D^P$ is simply connected for any $P\in\mathcal{P}(G)$,
\item $D$ satisfies the weak gap manifold condition,
  \item $\dim D^P\geq 5$ and $\dim D^H\geq 2$ for any $P\in\mathcal{P}(G)$ and $H\in\mathcal{PC}(G)$ ,
\item any pseudocyclic subgroup of $G$ occurs as an isotropy subgroup of the action of $G$ on $D$,
       \item $D^L$ is finite for any $L\in\mathcal{L}(G)$,
       \item $\dim D\geq 6$,
          \item $T_{y_0}D$ is $\mathcal{P}$-oriented for some $y_0\in D^G$, 
\end{enumerate}
where by the \emph{weak gap manifold condition} for a $G$-manifold $M$, we mean the following, \cite{Morimoto1998}.
\begin{itemize}[-]
    \item $\dim M^P\geq 2\dim M^H$
    \item if $\dim M^P=2\dim M^H$, then $M^H$ is connected, oriented, and the transformation $\theta_g:M^H\rightarrow M^H$ is orientation preserving for any element $g\in N_G(H)$, and $\dim M^H>\dim M^K+1$ whenever $H<K\leq G$
    \item if $\dim M^P-2\dim M^H=\dim M^{P'}-2\dim M^{H'}=0$, then the subgroup generated by subgroups $H$ and $H'$ is not a large subgroup of $G$,
\end{itemize}
where, for $H\leq G$, the expression $\dim M^H$ denotes the maximum dimension of connected components of $M^H$.

Taking the double of the disk $D$ allows us to construct a smooth action on a sphere $S$ with four fixed points and tangent spaces at these points isomorphic to two copies of $U$ and two copies of $V$ and the conditions $(1')-(7')$ (the assumption on simply connectedness still holds) are satisfied for this action. The next step is to make use of the Deleting-Inserting Theorem of Morimoto, \cite[Theorem 0.1]{Morimoto1998}, which allows to delete two fixed points with tangent spaces at these points one of the copies of $U$ and $V$. Thus, once we have two Smith matched and $\mathcal{P}$-oriented $\mathbb{R}G$-modules $U$ and $V$, then we can construct a smooth action on a sphere with two fixed points and tangent spaces at these points isomorphic to $U$ and $V$. 
\subsection{Ensuring that the induced modules are Smith equivalent}
\begin{lemma}\label{dimensionFormula}
If $K\leq G$ and $V$ is an $\mathbb{R}N$-module, then $\dim(\Ind_N^G(V)^K)=\frac{|G|}{|N|}\cdot\frac{|K\cap N|}{|K|}\dim V^{K\cap N}$.
\end{lemma}
\begin{proof}
Put $W=\Ind_N^{G}(V)$. Since $\dim W^K=\frac{1}{|K|}\sum_{k\in K}\chi_W(k)$, we have to compute the induced character, $\chi_W$. Since $N$ is normal in $G$ and $\Ind_N^G:\RO(N)\rightarrow\RO(G)$ is a monomorphism, we know by Theorem \ref{theorem:A} that the real conjugacy classes in $G$ of elements from $N$ coincide with their real conjugacy classes in $N$. Thus, from Theorem \ref{theorem:inducedCharacterFormula}, we conclude that
\begin{equation*}
    \chi_{\raisebox{-.5ex}{$\scriptstyle W$}}(g) = \begin{cases}
               0               & g \not\in N\\
               \frac{|G|}{|N|}\chi_{\raisebox{-.5ex}{$\scriptstyle V$}}(g)               & g\in N
               
           \end{cases}
\end{equation*}
Indeed, if for some $n\in N$ we have $(n)_G=(n)_N$, then the formula above follows directly from the orbit-stabilizer theorem and Theorem \ref{theorem:inducedCharacterFormula}. In case $(n)_N\subsetneq(n)_G$, it follows that $(n)^{\pm}_G=(n)_G=(n^{-1})_G=(n)\cup(n^{-1})_N$, since $(n^{\pm})_G=(n^{\pm})_N$.
Thus, by the orbit-stabilizer theorem and Lemma \ref{lemma:trace},
$$
\chi_{\raisebox{-.5ex}{$\scriptstyle W$}}(n)=|C_G(n)|\Big(\frac{\chi_{\raisebox{-.5ex}{$\scriptstyle V$}}(n)}{|C_N(n)|}+\frac{\chi_{\raisebox{-.5ex}{$\scriptstyle V$}}(n^{-1})}{|C_N(n^{-1})|}\Big)=\frac{|G|}{|N|}\chi_{\raisebox{-.5ex}{$\scriptstyle V$}}(n)
$$
Hence
\begin{align*}
\dim W^K&=\frac{1}{|K|}\sum_{k\in K\cap N}\chi_{\raisebox{-.5ex}{$\scriptstyle W$}}(k)=\frac{1}{|K|}\cdot\frac{|G|}{|N|}\sum_{k\in K\cap N}\chi_{\raisebox{-.5ex}{$\scriptstyle V$}}(k)\\
&=\frac{|G|}{|N|}\cdot\frac{|K\cap N|}{|K|}\dim V^{K\cap N}.
\end{align*}
\end{proof}
\begin{theorem}\label{theorem1}
If an $\mathbb{R}N$-module $V$ satisfies the conditions $(2)-(4)$, then $W=\Ind_N^{G}(V)$ satisfies these conditions as well as an $\mathbb{R}G$-module.
\end{theorem}
\begin{proof}
It is obvious that the induced module of a real module is real as well. For the proof of $(2)$, take $P<K\leq G$, $P$ - a $p$-group for some prime $p$. We must show $\dim W^P\geq 2\dim W^K$. Using Lemma \ref{dimensionFormula}, this is equivalent to
\begin{equation}\label{eq:1}
\frac{|P\cap N|}{|P|}\dim V^{P\cap N}\geq 2\frac{|K\cap N|}{|K|}\dim V^{K\cap N}
\end{equation}
We show that $\frac{|P\cap N|}{|P|}\geq \frac{|K\cap N|}{|K|}$. This is equivalent to $\frac{|K|}{|P|}\geq \frac{|K\cap N|}{|P\cap N|}$. For this, notice that we have an injective coset map
$$
f:(K\cap N)/(P\cap N)\rightarrow K/P
$$
$$
[k]_{P\cap N}\mapsto [k]_P
$$
Indeed, to see that $f$ is well-defined take $k_1,k_2\in K\cap N$ with $k_1k_2^{-1}\in P\cap N$. In particular, $k_1k_2^{-1}\in P$. For injectivity of $f$ assume $f([k_1]_{P\cap N})=f([k_2]_{P\cap N})$, which is equivalent to $k_1k_2^{-1}\in P$. However, since $k_1,k_2\in K\cap N$, this yields $k_1k_2^{-1}\in P\cap N$ and $[k_1]_{P\cap N}=[k_2]_{P\cap N}$. 

Thus, if $P\cap N<K\cap N$, the inequality \ref{eq:1} follows, since $(1)$ is satisfied for $V$ - in particular for $P\cap N$ and $K\cap N$, that is, $\dim V^{P\cap N}\geq\dim V^{K\cap N}$. If $P\cap N=K\cap N$, then, since $P<K$, $|K|\geq 2|P|$ and \ref{eq:1} follows as well.

Now, consider the condition $(3)$. Take $P\leq G$ a $p$-subgroup for some prime $p$ and $K\in \mathcal{PC}(G)$. From Lemma \ref{dimensionFormula}, we have $\dim W^P\geq 5\Leftrightarrow \frac{|G|}{|N|}\cdot\frac{|P\cap N|}{|P|}\dim V^{P\cap N}\geq 5$. Since $(3)$ is satisfied for $V$, $\dim V^{P\cap N}\geq 5$. Hence, we only have to show $\frac{|G|}{|N|}\cdot\frac{|P\cap N|}{|P|}\geq 1$, which is equivalent to $\frac{|G|}{|P|}\geq \frac{|G\cap N|}{|P\cap N|}$ and we can repeat the trick with the coset injection $f$. To prove $\dim W^K\geq 2$, we repeat the same argument.

Concerning $(4)$, let $K\leq G$ be pseudocyclic with $P\trianglelefteq K$ a $p$-group for some prime $p$ such that $K/P$ is cyclic. Note first that in this case, the map $f:(K\cap N)/(P\cap N)\rightarrow K/P$ becomes a group monomorphism. Therefore $(K\cap N)/(P\cap N)$ is also cyclic and thus $K\cap N$ pseudocyclic. Now, pick any $K'$ such that $K<K'\leq G$. We must show $\dim W^{K'}<\dim W^K$, which is equivalent by Lemma \ref{dimensionFormula} to 
$$
\frac{|K'\cap N|}{|K'|}\dim V^{K'\cap N}<\frac{|K\cap N|}{|K|}\dim V^{K\cap N}.
$$ 
As previously, $\frac{|K\cap N|}{|K|}\geq \frac{|K'\cap N|}{|K'|}$, and, since $K\cap N$ is pseudocyclic, the claim follows if $K'\cap N>K\cap N$. Otherwise, $\frac{|K\cap N|}{|K|}$ is strictly greater than $\frac{|K'\cap N|}{|K'|}=\frac{|K\cap N|}{|K'|}$ and the claim follows as well.
\end{proof}
\begin{lemma}\label{lemma:2}
If $K$ is a large subgroup of $G$, then $K\cap N$ is a large subgroup of $N$.
\end{lemma}
\begin{proof}
Assume $K$ is a large subgroup of $G$ and $\mathcal{O}^p(G)\leq K$ for some prime $p$. Let $L=\mathcal{O}^p(G)\cap N$. Then $L$ is a normal subgroup of $N$ and $N/L$ is a $p$-group. To see this, note that we have a quotient group monomorphism $N/L\rightarrow G/\mathcal{O}^p(G)$, from which it follows that $\frac{|N|}{|L|}$ divides $\frac{|G|}{|\mathcal{O}^p(G)|}$, which is a power of $p$. Thus, $N/L$ is a $p$-group.

Since $L$ is a normal subgroup of $N$ with $N/L$ a $p$-group, it follows that $\mathcal{O}^p(N)\leq L$. On the other hand, $\mathcal{O}^p(G)\leq K$, which implies $L\leq K\cap N$. Therefore, $\mathcal{O}^p(N)\leq K\cap N$.
\end{proof}
Let us now note the following auxiliary proposition.
\begin{proposition}\label{proposition:1}
Let $G$ be a group and $U$ and $V$ be two $\mathbb{R}G$-modules. Then $U-V\in PO(G)$ if and only if $\chi_{\raisebox{-.5ex}{$\scriptstyle U$}}(p)=\chi_{\raisebox{-.5ex}{$\scriptstyle V$}}(p)$ for any element $p\in G$ of prime power order.
\end{proposition}
\begin{proof}
Assume $U-V\in PO(G)$. By definition, this means that the restrictions, $\Res_P^G(U)$ and $\Res_P^G(V)$, of $U$ and $V$ to any subgroup $P\leq G$ of prime power order are isomorphic as $\mathbb{R}P$-modules. Take any $p\in G$ of prime power order. We must show that $\chi_{\raisebox{-.5ex}{$\scriptstyle U$}}(p)=\chi_{\raisebox{-.5ex}{$\scriptstyle V$}}(p)$. The cyclic group $P=\{1,p,\ldots,p^{|p|-1}\}$ is a subgroup of $G$ of prime power order. Thus, $\Res_P^G(U)\cong\Res_P^G(V)$, which is equivalent to saying that the characters of these restrictions take equal values for every $p\in P$. On the other hand, $\chi_{\raisebox{-.5ex}{$\scriptstyle \Res_P^G(U)$}}(p)=\chi_{\raisebox{-.5ex}{$\scriptstyle U$}}(p)$ and $\chi_{\raisebox{-.5ex}{$\scriptstyle \Res_P^G(V)$}}(p)=\chi_{\raisebox{-.5ex}{$\scriptstyle V$}}(p)$. Hence, $\chi_{\raisebox{-.5ex}{$\scriptstyle U$}}(p)=\chi_{\raisebox{-.5ex}{$\scriptstyle V$}}(p)$. 

Now, assume $\chi_{\raisebox{-.5ex}{$\scriptstyle U$}}(p)=\chi_{\raisebox{-.5ex}{$\scriptstyle V$}}(p)$ for any $p\in G$ of prime power order. Take any $P\leq G$ of prime power order. We must show that $\Res_P^G(U)\cong\Res_P^G(V)$. Note that every element $p\in P$ is of prime power order, which implies that $$\chi_{\raisebox{-.5ex}{$\scriptstyle U$}}(p)=\chi_{\raisebox{-.5ex}{$\scriptstyle V$}}(p)\Leftrightarrow \chi_{\raisebox{-.5ex}{$\scriptstyle \Res_P^G(U)$}}(p)=\chi_{\raisebox{-.5ex}{$\scriptstyle \Res_P^G(V)$}}(p)$$
Therefore $\Res_P^G(U)\cong\Res_P^G(V)$.
\end{proof}
Now, we can prove the following lemma ensuring the condition $(1)$ for the induced modules.
\begin{lemma}\label{lemma:3}
Let $H\leq G$. If $U-V\in\widetilde{PO}(H)$, then $\Ind_H^{G}(U)-\Ind_H^{G}(V)\in\widetilde{PO}(G)$.
\end{lemma}
\begin{proof}
Assume $U$ and $V$ satisfy the assumptions of the lemma. Note first that $U^H\cong(\Ind_H^{G}(V))^G$ and $V^H\cong(\Ind_H^{G}(V))^G$ by the construction of the induced module. This, in connection with $\dim U^H=\dim V^H=0$, gives us $\dim(\Ind_H^{G}(U))^G=\dim(\Ind_H^{G}(V))^G=0$. Thus, it suffices to show that $\Ind_H^{G}(U)-\Ind_H^{G}(V)\in PO(G)$, which is equivalent, by Proposition \ref{proposition:1}, to proving that $\Ind_H^{G}(\chi_{\raisebox{-.5ex}{$\scriptstyle U$}})(p)=\Ind_H^{G}(\chi_{\raisebox{-.5ex}{$\scriptstyle V$}})(p)$ for any $p\in G$ of prime power order. 

Take $p\in G$ of prime power order. If the intersection of the conjugacy class $(p)$ of element $p$ in $G$ with the subgroup $H$ is empty, then, by Theorem \ref{theorem:inducedCharacterFormula}, $\Ind_H^{G}(\chi_{\raisebox{-.5ex}{$\scriptstyle U$}})(p)=\Ind_H^{G}(\chi_{\raisebox{-.5ex}{$\scriptstyle V$}})(p)=0$. Assume $H\cap(p)$ is nonempty and let $h_1,\ldots,h_m$ be the representatives of all the distinct conjugacy classes of elements in $H$ of the set $H\cap(p)$. Then, by Theorem \ref{theorem:inducedCharacterFormula}, 
$$
\Ind_H^{G}(\chi_{\raisebox{-.5ex}{$\scriptstyle U$}})(p)=\Ind_H^{G}(\chi_{\raisebox{-.5ex}{$\scriptstyle V$}})(p)\Leftrightarrow
$$
\begin{eqnarray*}\frac{\chi_{\raisebox{-.5ex}{$\scriptstyle U$}}(h_1)}{|C_H(h_1)|}+\ldots+\frac{\chi_{\raisebox{-.5ex}{$\scriptstyle U$}}(h_m)}{|C_H(h_m)|}=\frac{\chi_{\raisebox{-.5ex}{$\scriptstyle V$}}(h_1)}{|C_H(h_1)|}+\ldots+\frac{\chi_{\raisebox{-.5ex}{$\scriptstyle V$}}(h_m)}{|C_H(h_m)|}
\end{eqnarray*}
The above equality is true since $h_i$ is of prime power order and thus, by our assumption, $\chi_{\raisebox{-.5ex}{$\scriptstyle U$}}(h_i)=\chi_{\raisebox{-.5ex}{$\scriptstyle V$}}(h_i)$ for $i=1,\ldots,m$.
\end{proof}
We can prove now Theorem \ref{theorem:B} announced in the Introduction.
\begin{proof}[Proof of Theorem \ref{theorem:B}]
Assume $U$ and $V$ are two non-isomorphic Smith matched $\mathbb{R}N$-modules such that $W_U=\Ind_N^{G}(U)$ and $W_V=\Ind_N^{G}(V)$ are $\mathcal{P}$-oriented. Since $U-V\in\widetilde{PO}(N)$, it follows by Lemma \ref{lemma:3} that $W_U-W_V\in\widetilde{PO}(G)$. By Theorem \ref{theorem1}, the conditions $(2)-(4)$ are satisfied for $W_U$ and $W_V$. We show that $W_U^L=W_V^L=\{0\}$ for any large subgroup of $L\in \mathcal{L}(G)$. We have, by Lemma \ref{lemma:2}, $L_N=L\cap N\in \mathcal{L}(N)$. Thus $U^{L_N}=\{0\}$ and $V^{L_N}=\{0\}$ by our assumption. This, in connection with Lemma \ref{dimensionFormula}, gives $W_U^L=W_V^L=\{0\}$. Obviously, $\dim(\Ind_N^G(U))=\dim(\Ind_N^G(V))\geq 6$. This completes the proof that $W_U$ and $W_V$ are Smith matched. Hence, by the considerations at the beginning of this section, there exists a two fixed point action on some sphere $S$ with tangent spaces at the fixed points isomorphic to $W_U$ and $W_V$. Moreover, since $U\not\cong V$ as $\mathbb{R}N$-modules, it follows that $W_U\not\cong W_V$ as $\mathbb{R}G$-modules for $\Ind_N^G:\RO(N)\rightarrow\RO(G)$ is a monomorphism.
\end{proof}
\begin{remark}\label{remark:porientability}
\emph{If we drop the assumption of $\mathcal{P}$-orientability of $W_U$ and $W_V$, we can formulate a similar assertion. By \cite[Key Lemma, pp. 887]{Solomon2002}, the modules $W_U'=W_U\oplus W_U$ and $W_V'=W_V\oplus W_U$ are $\mathcal{P}$-oriented. Obviously, these modules are Smith matched. Thus, there exists a two fixed point action of $G$ on some sphere $S$ with tangent spaces at the two fixed points isomorphic to $\Ind_N^G(U\oplus U)$ and $\Ind_N^G(U\oplus V)$. Obviously, these modules are not isomorphic.} 
\end{remark}
\begin{remark}\label{remark:directProducts}
\emph{In case $N$ is a direct summand of $G$ with $G=N\times H$, then, once we have two non-isomorphic Smith matched $\mathbb{R}N$-modules, then there exists a Smith exotic action of $G$ on $S^{2n|H|}$. This is a consequence of the fact that the condition $(1)$ (and thus $(2)$) from Corollary \ref{corollary:inducedNormal} concerning concjugacy classes is naturally satisfied if $N$ is a direct summand.}

\end{remark}
\subsection{Application to direct products with $\SL(2,5)$ as a direct summand}
For $N=\SL(2,5)$, we find the Smith matched and $\mathcal{P}$-oriented $\mathbb{R}N$-modules $U$ and $V$ with the minimal dimension. 

The following table contains nontrivial characters of real irreducible representations of $N$, see \cite[pp. 10-11]{Mizerka2019_2}.
\begin{center}
\begin{tabular}{l|l*{7}{c}r}
            &  $\mathbf{1}$& $z$ & $c$ & $d$ & $zc$&$zd$&$a$  & $b$ &$b^2$\\
\hline
$V_{3,1}$     & $3$ & $3$ & $\frac{1+\sqrt{5}}{2}$ & $\frac{1-\sqrt{5}}{2}$ &$\frac{1+\sqrt{5}}{2}$ & $\frac{1-\sqrt{5}}{2}$ & $-1$ & $0$&$0$ \\
$V_{3,2}$     & $3$ & $3$ & $\frac{1-\sqrt{5}}{2}$ & $\frac{1+\sqrt{5}}{2}$ & $\frac{1-\sqrt{5}}{2}$ & $\frac{1+\sqrt{5}}{2}$ &$-1$ & $0$&$0$ \\
$V_{4,1}$    & $4$ & $4$ & $-1$ & $-1$ &$1$ & $1$ &$ 0$ & $-1$ &$-1$ \\
$V_{4,2}$    &$ 4 $&$ -4 $&$ -1+\sqrt{5} $&$ -1-\sqrt{5} $& $ 1-\sqrt{5} $&$ 1+\sqrt{5} $&$0$ & $2$&$-2$ \\
$V_{4,3}$    &$ 4 $&$ -4 $&$ -1-\sqrt{5} $&$ -1+\sqrt{5} $& $ 1+\sqrt{5} $&$ 1-\sqrt{5}$&$0$ & $2$&$-2$ \\
$V_5$   & $5$ & $5$ & $0$ & $0$ & $0$&$0$&$1$ & $-1 $ &$-1$\\
$V_8$    & $8$ & $-8$ & $-2$ & $-2$ &$2$ & $2$ &$ 0$ & $-2$&$2$ \\
$V_{12}$   & $12$ & $-12$ & $2$ & $2$ &$-2$ & $-2$ & $0$ &$ 0$  &$0$
\end{tabular}
\end{center}
The subgroup lattice of $N$ (up to conjugacy) is as follows.
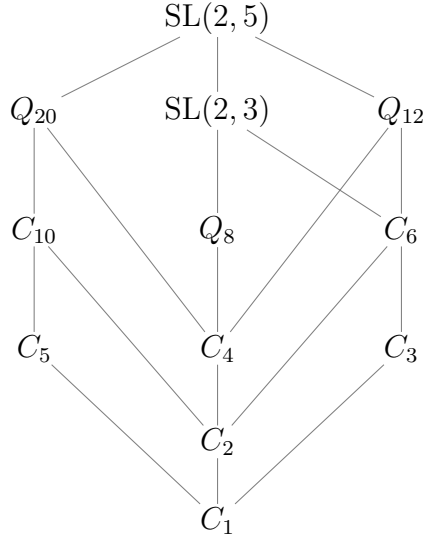
\begin{center}
\begin{tikzpicture}[scale=1.3,sgplattice]
  \node at (2,0) (1) {\gn{C1}{C_1}};
  \node at (2,0.803) (2) {\gn{C2}{C_2}};
  \node at (3.88,1.76) (3) {\gn{C3}{C_3}};
  \node at (0.125,1.76) (4) {\gn{C5}{C_5}};
  \node at (2,1.76) (5) {\gn{C4}{C_4}};
  \node at (3.88,2.97) (6) {\gn{C6}{C_6}};
  \node at (0.125,2.97) (7) {\gn{C10}{C_{10}}};
  \node at (2,2.97) (8) {\gn{Q8}{Q_8}};
  \node at (3.88,4.19) (9) {\gn{Q12}{Q_{12}}};
  \node at (0.125,4.19) (10) {\gn{Q20}{Q_{20}}};
  \node at (2,4.19) (11) {\gn{SL(2,3)}{\SL(2,3)}};
  \node at (2,5.14) (12) {\gn{SL(2,5)}{\SL(2,5)}};
  \draw[lin] (1)--(2) (1)--(3) (1)--(4) (2)--(5) (2)--(6) (3)--(6) (2)--(7)
     (4)--(7) (5)--(8) (5)--(9) (6)--(9) (5)--(10) (7)--(10) (6)--(11)
     (8)--(11) (9)--(12) (10)--(12) (11)--(12);
  
\end{tikzpicture}
\end{center}
where $Q_{4k}=\langle a^{2k}=1,b^2=a^k,bab^{-1}=a^{-1}\rangle$. By \cite[pp.19-20]{Mizerka2019_2} and the computations for non-cyclic subgroups of $N$ (see author's implementations \cite{Mizerka2019}, using the GAP software \cite{GAP4}), the fixed point dimension table for real irreducible representations of $G$ is as follows (dimensions for $V_{4,2}$ and $V_{4,3}$, as well as $V_{3,1}$ and $V_{3,2}$ agree).
\begin{center}
\small\addtolength{\tabcolsep}{0pt}
\begin{tabular}{l|c*{12}{c}}
            &  $C_{1}$& $C_{2}$ & $C_{3}$ & $C_{4}$ & $C_{5}$  & $C_{6}$ & $C_{10}$&$Q_{8}$ & $Q_{12}$&$Q_{20}$&$\SL(2,3)$&$\SL(2,5)$  \\
\hline
$V_{3,1}$ & $3$ & $3$ & $1$ & $1$ &  $1$& $1$ &$1$ & $0$& $0$& $0$&$0$ & $0$\\
$V_{4,1}$   & $4$ & $0$ &  $0$&  $0$& $0$ &  $0$&$0$ & $0$& $0$& $0$& $0$&$0$ \\
$V_{4,2}$     &  $4$& $4$ & $2$ & $2$ & $0$ & $2$ & $0$& $1$&$1$ & $0$& $1$&$0$ \\
$V_5$     & $5$ & $5$ & $1$ & $3$ & $1$ & $1$ & $1$& $2$&$1$ &$1$ &$0$ &$0$ \\
$V_8$     &  $8$& $0$ & $4$ & $0$ & $0$ & $0$ &$0$ & $0$& $0$& $0$& $0$& $0$\\
$V_{12}$     & $12$ & $0$ & $4$ & $0$ & $4$ & $0$ &$0$ & $0$& $0$& $0$&$0$ &$0$ \\
\end{tabular}
\end{center}

We checked with GAP (see \cite{Mizerka2019}) that $\mathbb{R}N$-modules $U=2V_{3,1}\oplus V_{4,2}\oplus 2V_{4,3}\oplus 2V_5\oplus V_8\oplus V_{12}$ and $V=2V_{3,2}\oplus V_{4,1}\oplus 2V_{4,3}\oplus 2V_5\oplus V_8\oplus V_{12}$ are Smith matched and $\mathcal{P}$-oriented. Hence, there exists a Smith exotic action of $G$ on $S^{48}$. This yields, by Remark \ref{remark:directProducts}, the existence of Smith exotic actions on $S^{96|H|}$ of $N\times H$, where $H$ is any group.
\section{Examples for which Conjecture \ref{conjecture:C} holds}
Assume $G$ is an Oliver group. Before, we proceed, let us recall the \emph{Laitinen conjecture}. 
\begin{conjecture}\cite[Appendix]{LaitinenPawalowski1999}\label{conjecture:Laitinen}
If $\prim(G)\geq 2$, then $\Sm(G)\neq 0$.
\end{conjecture}
Following \cite[Answer. 3.9.1]{Pawalowski2018}, the Laitinen conjecture holds in either of the following situations.
\begin{enumerate}[(a)]
    \item $G$ is of odd order,
    \item $G$ has a cyclic quotient of composite odd order (in particular, $G$ is nilpotent, see \cite[p. 25]{Pawalowski2018}),
    \item $G$ is nonsolvable and $G\neq \Aut(A_6)$.
\end{enumerate} 
Thus, in either of the three cases above, the triviality of Smith sets is easy to verify by means of the Laitinen conjecture. What is interesting is to provide new examples of Oliver groups which do not lie in either of the three classes described above. We apply Lemma \ref{dimensionFormula} to give a new example of $G$ with nontrivial Smith set such that $G$ does not belong to $(a)$, $(b)$ or $(c)$. Note that if $G$ belongs to either of the classes $(a)$, $(b)$ or $(c)$ and $N\trianglelefteq G$ such that $\prim(N)\geq 2$ , then Conjecture \ref{conjecture:C} holds. Indeed, since $\prim(N)\geq 2$, then if $\Ind_N^G:\RO(N)\rightarrow\RO(G)$ is a monomorphism, we have $\prim(G)\geq 2$. Thus, $\Sm(G)\neq 0$ by the Laitinen conjecture applied for $G$.

Let us call $G$ a \emph{gap group} if $\mathcal{P}(G)\cap\mathcal{L}(G)=\emptyset$ and there exists an $\mathbb{R}G$-module $V$ such that $\dim V^L=0$ for every $L\in\mathcal{L}(G)$ and $V$ satisfies the strong gap condition. Following \cite{PawalowskiSumi2009}, we introduce a symbol $\widetilde{PO}(G)^{\mathcal{L}}$ for a subgroup of $\widetilde{PO}(G)$ consisting of differences $U-V$ such that $\dim U^L=\dim V^L=0$ for any $L\in\mathcal{L}(G)$.
\begin{lemma}\label{lemma:applicationGap}
Assume $\widetilde{\PO}(G)^{\mathcal{L}}\neq 0$ and $N$ is a normal gap subgroup of $G$ such that $\Ind_N^G:\RO(N)\rightarrow\RO(G)$ is a monomorphism. Then $\Sm(G)\neq 0$.
\end{lemma}
\begin{proof}
Since $N$ is a gap group, there exists an $\mathbb{R}G$-module $V$ satisfying the gap condition and such that $\dim V^L=0$ for any $L\in \mathcal{L}(N)$. As in the proofs of Theorem \ref{theorem1} and Theorem \ref{theorem:B} that $\Ind_N^G(V)$ satisfies the strong gap condition and $\dim(\Ind_N^G(V))^L=0$ for any $L\in\mathcal{L}(G)$. Moreover, $\mathcal{P}(G)\cap\mathcal{L}(G)=\emptyset$ since $\mathcal{P}(N)\cap\mathcal{L}(N)=\emptyset$. Indeed if $H\in\mathcal{P}(G)\cap\mathcal{L}(G)$, then it follows by Lemma \ref{lemma:2} that $H\cap N\in\mathcal{L}(N)$. On the other hand $H\cap N\in\mathcal{P}(N)$. A contradiction. This proves that $G$ is a gap group and we cocnlude from \cite[Theoem 4.1]{PawalowskiSumi2009} that $\widetilde{\PO}(G)^{\mathcal{L}}\in\Sm(G)$. Since $\widetilde{\PO}(G)^{\mathcal{L}}\neq 0$ by our assumption, the assertion follows.
\end{proof}
Following \cite{PawalowskiSumi2009}, let us denote for $H\trianglelefteq G$ by $\overline{\prim}(G,H)$ the number of real conjugacy classes in $G/H$ of elements of the form $gH$ containing elements of $G$ not of prime power order. (note that $\prim(G)=r_G$ and $\overline{\prim}(G,H)=\overline{r}_{G/H}$ in \cite{PawalowskiSumi2009}). Let $G^{nil}$ denote the smallest normal subgroup of $G$ such that $G/G^{nil}$ is nilpotent. The rank of $\widetilde{\PO}(G)^{\mathcal{L}}\neq 0$ can be then estimated from below by $\prim(G)-\overline{\prim}(G,G^{nil})$, see \cite[Lemma 4.3]{PawalowskiSumi2009}.

Take $G=C_6\times A_4\times D_{30}\cong C_{30}\rtimes(S_3\times A_4)$, the direct product of cyclic group of order $6$, alternating group of degree $4$ and dihedral group of order $30$. $G$ is an example of an Oliver group which does not belong to either of the classes $(a)$, $(b)$ or $(c)$ (the fact that $G$ is Oliver follows since $S_3\times A_4$ is Oliver, see \cite{PawalowskiSumi2009}). Assume $a$ is the generator of $C_6\leq G$ and $D_{30}\leq G$ has the standard presentation $\langle x,y|x^{15}=1,y^2=1,yxy=x^{-1}\rangle$. Put $N=\langle a\rangle\times\{1\}\times\langle x\rangle\cong C_{15}$, the cyclic group of order $30$. Obviously $N$ is normal in $G$. Take any $n=(a^k,1,x^l)\in N$, where $k=0,...,5$ and $l=0,...,14$. Notice that $(n)^{\pm}_N=(n)^{\pm}_G$. Indeed, for any $n'\in(n)_G^{\pm}$ we can find $m=0,...,14$ such that 
$$n'=(a^{\pm k},1,(yx^m)x^{\pm l}(yx^m)^{-1})=(a^{\pm k},1,yx^{\pm l}y^{-1})=(1,1,y)n^{\pm}(1,1,y)^{-1}$$
Thus $n'\in(n)^{\pm}_G$ and $(n)^{\pm}_N=(n)^{\pm}_G$. It follows then by Corollary \ref{corollary:inducedNormal} that $\Ind_N^G:\RO(N)\rightarrow\RO(G)$ is a monomorphism. Moreover, $N=C_{30}$ is a gap group (in fact, the gap group with the minimal order, see \cite[p. 330]{Morimoto2000}). We show that $\widetilde{\PO}(G)^\mathcal{L}\neq 0$ which would imply that $\Sm(G)\neq 0$ by Lemma \ref{lemma:applicationGap}. We have $G^{nil}=\{1\}\times H\times \langle x\rangle$, where $H$ is a subgroup of $A_4$ isomorphic to $C_2^2$. Hence, $G/G^{nil}\cong C_6^2$. Using GAP \cite{GAP4} we compute the primary number of $G$, $\prim(G)=107$ and $\beta'(G/G^{nil})=\beta'(C_6^2)=20$. Thus
$$
\rank(\widetilde{\PO}(G)^{\mathcal{L}})\geq \prim(G)-\beta'(G/G^{nil})=87>0
$$
In particular, the assertion of Conjecture \ref{conjecture:C} for $G$ and $N$ holds. 

Let us show some examples of solvable Oliver groups which are not nilpotent and confirm Conjecture \ref{conjecture:C}. These groups are 
\begin{itemize}
    \item $G_1=C_3\times S_4=\SmallGroup(72,42)$, 
    \item $G_2=S_3\times A_4=\SmallGroup(72,44)$, 
    \item $G_3=(C_2^2\rtimes C_3)^2\rtimes C_2=\SmallGroup(288,1025)$,
    \item $G_4=A_4^2\rtimes C_2^2=\SmallGroup(576,8654)$.
    \end{itemize} Denote by $\mathcal{N}(G)$ the set of normal subgroups $N$ of $G$ (up to isomorphism) such that $N\not\in\{1,G\}$ and $\Ind_N^G:\RO(N)\rightarrow\RO(G)$ is a monomorphism. We verify in GAP that $\mathcal{N}(G_1)=\{C_3,A_4,S_4\}$, $\mathcal{N}(G_2)=\{C_3,S_3,A_4\}$, $\mathcal{N}(G_3)=\emptyset$,  $\mathcal{N}(G_4)=\{G_3\}$. By \cite[Answer 4.1.3 (1)]{Pawalowski2018} we have $\Sm(C_3)=0$, while $\Sm(S_3)=\Sm(S_4)=\Sm(A_4)=0$ by \cite[Corollary 4.5.4 (3)]{Pawalowski2018} for $\prim(S_3)=\prim(S_4)=\prim(A_4)=0$. The triviality of $\Sm(G_3)$ was proved in \cite[Proposition 5.4]{PawalowskiSumi2009}. Thus, the Smith set of every group from $\mathcal{N}(G_i)$, $i=1,...,6$ is trivial, see \cite{Pawalowski2018}. This, in particular, confirms Conjecture \ref{conjecture:C} for any $N\trianglelefteq G_i$ with $N\in\mathcal{N}(G_i)$ for $i=1,...,4$.

Similar result holds for the unique nonsolvable group $G$ for which the Laitinen conjecture does not hold, that is for $G=\Aut(A_6)$. We have $\mathcal{N}(G)=\{\SmallGroup(720,765)\}$. Put $N=\SmallGroup(720,765)$. Using GAP, we verify that $N$ is a nonsolvable group with $\prim(N)=0$. Thus, by the Laitinen conjecture, we have $\Sm(N)=0$ and the assertion of Conjecture \ref{conjecture:C} holds in this case.
 \section*{Acknowledgements}
 The author would like to thank Prof. Krzysztof Pawałowski for his comments and interest in the results presented here. I am also grateful to Dr. Marek Kaluba for remarks which substantially improved the presentation of this paper. Also, I would like to thank all the participants of the algebraic topology seminar held at Adam Mickiewicz University for helpful comments.


\newpage
\bibliographystyle{acm}
\bibliography{sample.bib}
\textcolor{white}{fsf}\\
\emph{Faculty of Mathematics and Computer Science}\\
\emph{Adam Mickiewicz University in Poznań}\\
\emph{ul Uniwersytetu Poznańskiego 4}\\
\emph{61-614 Poznań, Poland}\\
\emph{Email address:} piotr.mizerka@amu.edu.pl







\end{document}